\documentclass[twocolumn]{IEEEtran}
\usepackage{color}
\usepackage[ansinew]{inputenc}
\usepackage{multirow}
\usepackage{amsmath}
\usepackage{amssymb}
\usepackage{bm}
\usepackage{subfigure}
\usepackage{multirow}
\usepackage[usenames,dvipsnames]{pstricks}
\usepackage{epsfig}
\usepackage{pst-grad} % For gradients
\usepackage{pst-plot} % For axes
\usepackage{graphicx,color,psfrag}
\newcommand{\Prb}{\boldsymbol{\Pr}}
\newcommand{\E}{\boldsymbol{\mathrm{E}}}
\newcommand{\Pm}{\mathcal{P}}
\newcommand{\Qm}{\mathcal{Q}}
\newcommand{\Gm}{\mathcal{G}}
\newcommand{\Rm}{\mathcal{R}}

\newtheorem{thm}{Theorem}

\newtheorem{assum}{Assumption}
\newtheorem{rem}{Remark}
\newtheorem{lem}{Lemma}
\newtheorem{exmp}{Example}

\begin{document}

%\allowdisplaybreaks
%
%
%\begin{frontmatter}
%
\title{Performance vs complexity trade-offs for Markovian networked jump estimators}
\author{D. Dolz, D. E. Quevedo, I. Peñarrocha, and R. Sanchis% <-this % stops a space
\thanks{D. Dolz , I. Peñarrocha, and R. Sanchis are with Department of Industrial System Engineering and Design, Universitat Jaume I of Castelló, Spain {\tt\small \{ddolz,ipenarro,rsanchis\}@uji.es}}
\thanks{D. E. Quevedo is with the School of Electrical Engineering and Computer Science, The University of Newcastle, NSW, Australia {\tt\small dquevedo@ieee.org}}
}
%
%\thanks[footnoteinfo]{This work has been funded by MICINN project number DPI2011-27845-C02-02, and grants PREDOC/2011/37 and  E-2013-02 from \textit{Universitat Jaume I}}
%
%%\author[prnc]{Daniel Dolz}\ead{ddolz@uji.es},
%%\author[prnc]{Ignacio Peñarrocha}\ead{ipenarro@uji.es},
%%\author[prnc]{Roberto Sanchis}\ead{rsanchis@uji.es}
%%\address[prnc]{Departament d'Enginyeria de Sistemes Industrials i Disseny, Universitat Jaume I, Castelló}
%
%\author[First]{Daniel Dolz}
%\author[Second]{Daniel E. Quevedo}
%\author[First]{Ignacio Pe\~narrocha}
%\author[First]{Roberto Sanchis}
%
%
%\address[First]{Department of Industrial Systems Engineering and Design, University Jaume I, Castelló, Spain (e-mail: \{ddolz,ipenarro,rsanchis\}@uji.es)}
%\address[Second]{School of Elexctrical Engineering and Computer Science, The University of Newcastle, NSW, Australia (e-mail: dquevedo@ieee.org)}
%%\address[Second]{Colorado State University, Fort Collins, CO 80523 USA (e-mail: author@lamar. colostate.edu)}
%%\address[Third]{Electrical Engineering Department, Seoul National University, Seoul, Korea, (e-mail: author@snu.ac.kr)}
\maketitle

\begin{abstract}
This paper addresses the design of a state observer for networked systems with random delays and dropouts. The model of plant and network covers the cases of multiple sensors, out-of-sequence and buffered measurements. The measurement outcomes over a finite interval model the network measurement reception scenarios, which follow a Markov distribution. We present a tractable optimization problem to precalculate off-line a finite set of gains of jump observers. The proposed procedure allows us to trade the complexity of the observer implementation for achieved performance.
Several examples illustrate that the on-line computational cost of the observer implementation is lower than that of the Kalman filter, whilst the performance is similar.
\end{abstract}
%\end{frontmatter}

\section{Introduction}
Networked control systems are control systems where the information (output measurements and/or control inputs) is transmitted via a shared network. The use of networks reduces the installation cost and increases the flexibility, but leads to several network-induced effects such as time delays and packet dropouts (see~\cite{hespanha2007survey} and \cite{chen2011guest}). Control and estimation through a network must overcome these problems.

Considering the estimation problem, Kalman filter based solutions may give optimal performance, but at the expense of significant on-line computational complexity. The observer gain is time varying and must be computed online, even for linear time invariant systems (e.g.~\cite{liu2004kalman},~\cite{Sinopoli2004} and~\cite{Schenato08}). This motivates the search for computationally low cost alternatives.
In particular, the use of precalculated gains decreases the need of the implementation computing capacity, but increases the estimation error and requires both storage and a mechanism to choose the appropriate gain at each instant (e.g.~\cite{Smith2003,Sahebsara07,Penya12} and~\cite{Han2013}). The jump linear estimator approach proposed in ~\cite{Smith2003} improves the estimation with a set of precalculated gains which are chosen depending on the history of measurement availabilities. A better performance is achieved at the cost of increasing the estimator complexity in terms of storage requirements and gain selection mechanism. An approach of intermediate complexity is presented in recent work~\cite{Han2013} where the authors propose a gain dependency on the possible instant and arrival delay for each measurement in a finite set. Computing the gains off-line takes advantage from prior statistical knowledge about the network behavior. When the network behaves as a Markov chain, the design uses the transition probabilities (\cite{Smith2003} and \cite{Han2013}).

In this paper we face the estimator design problem for multisensor systems and networks with induced unbounded time-varying delays with known distribution. We derive a finite measurement outcomes parameter that models the network effects and follows a finite Markov chain. Based on this process, we propose a jump linear estimator that gives favorable trade-offs between on-line computational burden and estimation performance. Furthermore, we analyze the effects of reducing the number of stored gains (i.e., complexity) by means of sharing the use of each gain for different values of the finite measurement outcomes parameter.

Two are the main contributions of our current work with respect to~\cite{Smith2003} and~\cite{Han2013}. First, we consider the multisensor with multiple delays scenario. Second, we introduce a flexible way to handle different strategies for the gain dependency to find a compromise between implementation cost and estimation performance. Moreover, the %Markov chain
measurement reception model derived here allows to handle more complex gain observer dependencies that cannot be included in~\cite{Han2013}. The present work differs from our recent manuscript~\cite{Penya12} mainly in the consideration of the stochastic network behavior with unbounded consecutive dropouts instead of a deterministic approach.

The paper has the following structure. In Section~\ref{sec:prbapr} we describe the process, model the network effects, present the observer algorithm and derive estimation error expressions. In Section~\ref{sec:obsvdsgn} we develop the observer design, and demonstrate its convergence. In Section~\ref{sec:trade} we show how gain grouping approaches can be used to find a compromise between implementation cost and performance. Simulation studies are given in Section~\ref{sec:ej}, and Section~\ref{sec:conclu} draws conclusions.

\section{Problem approach}\label{sec:prbapr}
Let us consider linear time invariant discrete-time systems of the form
\begin{align}\label{estados}
&x[t+1]=A\,x[t]+B_u\,u[t]+B_w\,w[t],\\
&y_s[t]=c_s\,x[t]+v_s[t],
\end{align}
where $x\in\mathbb{R}^n$ is the state, $u\in\mathbb{R}^{n_u}$ is the control input, $y_s\in\mathbb{R}$ is the $s$-th measured output ($s=1,\ldots,n_y$) with $y[t]=\begin{bmatrix}y_1[t]&\ldots&y_{n_y}[t]\end{bmatrix}^T$, $w\in\mathbb{R}^{n_{w}}$ is the state disturbance modeled as a white noise signal of zero mean and known covariance $\E\{w[t]\,w[t]^T\}=W$, and $v_s\in\mathbb{R}$ is the $s$-th sensor noise assumed as an independent zero mean white noise signal with known variance $\E\{{v_{s}[t]}^2\}=\sigma_s^2$. Throughout this work we assume that the control input is causally available at all times, see Fig.~\ref{fig_problem}.

\begin{figure}[h]
\begin{center}
  \includegraphics[width=\linewidth]{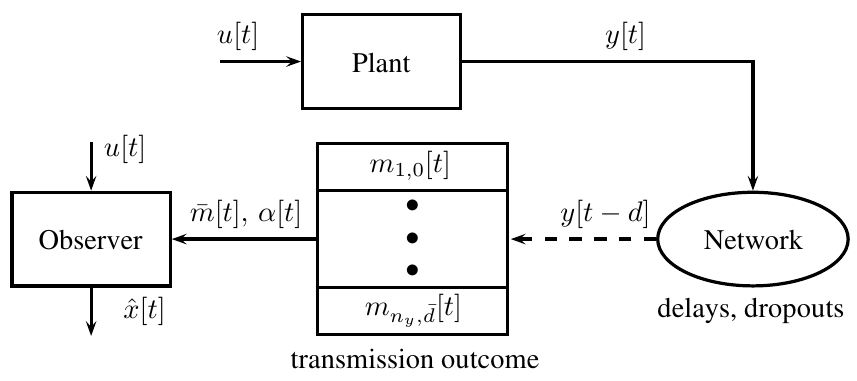}
\caption{Networked state estimator.}\label{fig_problem}
\end{center}
\end{figure}

Let us assume that samples from several sensors are taken synchronously with the input update and sent independently to the estimator unit through a network with packet dropouts and induced time-varying delays (see Fig.~\ref{fig_problem}). Let us assume synchronization between sensors and the estimator unit and time-tagged message sending. We denote as $\tau_s[t]\in\mathbb{N}$ the induced delay on the delivery of the $t$-th sample of sensor $s$, where $\tau_s[t]=\infty$ represents a measurement loss. We assume that the delays are bounded by $\bar d$; otherwise, we discard the measurement. Then, the network induced delay for all sensors can take values in a finite range $\tau[t]\in\{0,1,\ldots,\bar{d}\}\bigcup\infty$. The available information at instant $t$ at the estimator unit is the pair $(m_{s,d}[t],\alpha_{s,d}[t])$ for all  $s=1,\ldots,n_y$, $d=0,1,\ldots,\bar{d}$ being the induced delay, where
\begin{equation}\label{ec:m_s_ds}
m_{s,d}[t]=\alpha_{s,d}[t]\,y_s[t-d],
\end{equation}
and
\begin{equation}\label{def:alpha_d_i}
\alpha_{s,d}[t]=\begin{cases}
   1 & \text{if } y_{s}[t-d] \text{ is received at instant } t \\
   0 & \text{otherwise }
  \end{cases}
 \end{equation}
Note that $\alpha_{s,d}[t]=1$ represents that the induced delay of the measurement from sensor $s$ sampled at instant $t-{d}$ is $\tau_s[t-d]=d$. We consider $m_{s,d}[t]=0$ if $y_s[t-d]$ does not arrive at instant $t$.

We introduce an aggregated model to deal with delayed measurements. %, which avoids running backwards the model,. % If it is not the case, we discard measurements with delays higher than $\bar{d}$. % to accomplish with this assumption.
The aggregated model including the delayed states is
\begin{equation}\label{estadosaug}
\bar{x}[t+1]=\bar{A}\bar{x}[t]+\bar{B}_uu[t]+\bar{B}_ww[t],
\end{equation}
where $\bar{x}[t]=\begin{bmatrix}x[t]^T&\cdots&x[t-\bar d]^T\end{bmatrix}^T$ and
\[
\bar{A}= \begin{bmatrix}A&0&\cdots&0\\I&\cdots&0&\vdots\\\vdots&\ddots&\vdots&\vdots\\0&\cdots&I&0\end{bmatrix},
\quad \bar{B}_u=\begin{bmatrix}B_u\\0\end{bmatrix},\quad \bar{B}_w=\begin{bmatrix}B_w\\0
\end{bmatrix},
\]
%with the received measurement equation
%\begin{equation*}
%m_s[t]=\alpha^{\tau_{s}[t]}_{s}[t]\,\biggl(\underbrace{[0_{1\times n\cdot\tau_{s}[t]}\; c_{s}\;0_{1\times n\cdot(\bar d-\tau_{s}[t])}]}_{\triangleq c_s[t]}\,\bar{x}[t]+v_{s}[t]\biggr),
%\end{equation*}
%where the case $\tau_{s}[t]=\infty$  is not taken into account in the previous model.
We incorporate $\bar{d}$ additional fictitious sensors for each actual sensor with a different constant delay, and express the available measurements from real sensor $s$ as
\begin{equation}\label{medidasaug}
\bar{m}_{{s}}[t]=[m_{s,0}[t]\,\cdots\,m_{s,\bar{d}}[t]]^T,
\end{equation}
with $m_{s,d}[t]$ as defined in~\eqref{ec:m_s_ds}. With that, the number of total (real and fictitious) sensors is $\bar{n}_y=n_y(1+\bar{d})$.
%This construct handles out-of-sequence samples, received packets including measurements from one sensor sampled on different instants (i.e., buffered measurements).%, and avoids the use of time varying matrices.
This model handles out-of-sequence and buffered samples (see~\cite{Penarrocha2012IJSS}).

\subsection{Network modeling}
Let us define the process $\theta[t]$ which captures the measurement transmission outcomes at times $\{t-\bar{d},\ldots,t\}$ as follows:
\begin{equation}
\theta[t]=\begin{bmatrix}\theta_1[t] & \cdots & \theta_{n_y}[t]\end{bmatrix}^T,
\end{equation}
with $\theta[t]$ a binary column vector of length $n_\theta=\frac{(\bar{d}+1)(\bar{d}+2)}{2}n_y$ and where $\theta_s[t]$ represents the measurement reception at times $\{t-\bar{d},\ldots,t\}$ from sensor $s$ (cf.~\cite{Han2013})
\begin{align}
\theta_s[t]&=\begin{bmatrix}\theta_{s,0}[t]&\cdots&\theta_{s,\bar{d}}[t]\end{bmatrix},\\
\theta_{s,d}[t]&=\begin{bmatrix}\alpha_{s,0}[t-d] & \alpha_{s,1}[t-d+1] & \cdots & \alpha_{s,d}[t]\end{bmatrix}.
\end{align}
$\theta_{s,d}[t]$ represents the transmission outcome of measurement $y_s[t-d]$ at times $\{t-\bar{d},\ldots,t\}$. In real-time systems, measurement $y_s[t-d]$ can only be received once. This implies that $\|\theta_{s,d}[t]\|_1\leq1$ with $\|\theta_{s,d}[t]\|_1=\sum_{d=0}^{\bar{d}}\alpha_{s,d}[t-\bar{d}+d]$. Clearly, $\theta[t]$  is an ergodic\footnote{In an ergordic Markov chains every state can be reached from every state in a finite time.}Markov chain (see~\cite{bremau99}) that can take values in the finite set
\begin{equation}\label{eq:Deltak}
{\theta}[t]\in\Theta=\{\vartheta_0,\vartheta_1,\ldots,\vartheta_{r}\},\,\,r={((\bar{d}+2)!)}^{n_y}-1,
\end{equation}
and where $\vartheta_i$ (for $i=0,\ldots,r$) denotes each possible combination of the historical measurement transmission outcomes. $\vartheta_0=0$ denotes the case where neither of the samples from $t-\bar{d}$ to $t$ is received. To obtain the transition probabilities, we use the following assumption.
\begin{assum}\label{asum:iid}
The delays $\tau_{s}[t]$ are i.i.d random variables with $\beta_{s,d}\triangleq\Prb\{\tau_s[t]=d\}$\footnote{Note that $\Prb\{\alpha_{s,d}[t]=1\}=\Prb\{\tau_s[t]=d\}=\beta_{s,d}$.}  for $d=0,\ldots,\bar{d}$, and where $\sum_{d=0}^{\bar{d}}\beta_{s,d}\leq1$.$\hfill\square$
\end{assum}
Let us denote the tail probabilities as $\bar{\beta}_{s,d}=\Prb\{\tau_s[t]> d\}=\Prb\{\|\theta_{s,d}[t]\|_{1}=0\}$.  Using Assumpion~\ref{asum:iid}, the elements from the transition probability matrix $\Lambda=[p_{i,j}]$ with $p_{i,j}=\Prb\{\theta[t+1]=\vartheta_j\big{|}\theta[t]=\vartheta_i\}$ are calculated as
\begin{equation}\label{ec:p_ij_theta}
\prod_{s=1}^{n_y}g(t+1,\bar{d},s)/g(t,\bar{d}-1,s),
\end{equation}
where
\begin{equation}\label{ec:p_ij_theta_2}
g(t,\bar{d},s)=\displaystyle{\prod_{\substack{{d}=0\\ d:\, \|\theta_{s,d}[t+1]\|_{1}=0}}^{\bar{d}}\bar{\beta}_{s,d}}\, \displaystyle{\prod_{\substack{{d}=0\\ d:\, \alpha_{s,d}[t+1]=1}}^{\bar{d}}{\beta}_{s,d}}.
\end{equation}
Equation~\eqref{ec:p_ij_theta} is only valid for feasible transitions. A transition is feasible if $\alpha_{s,d}[t-h]$ have the same value in both $\theta[t]$ and $\theta[t+1]$, for all $s=1,\ldots,n_y$, $d=0,\ldots,\bar{d}-1$ and $h=0,\ldots,\bar{d}-1-d$.
%$\theta_{s,h+1}^{d+1}[t+1]=\theta_{s,h+1}^d[t]$ for all $s=0,\ldots,n_y$, $h=0,\ldots,d$ and $d=0,\ldots,\bar{d}$, where $\theta_{s,h+1}^d[t]$ denotes the $(h+1)$-th element of $\theta_{s}^d[t]$.

Moreover, we denote the total probability of being at a given state $i$ as $\pi_i[t]=\Prb\{\theta[t]=\vartheta_i\}$, where $\pi[t]=[\pi_1[t],\ldots,\pi_{r}[t]]$ and $\pi[t+1]=\pi[t]\Lambda$.

Let us now define the measurement availability matrix at instant $t$ as
\begin{equation}\label{def:alpha}
\alpha[t]=f(\theta[t])=\bigoplus_{s=1}^{n_y}\left(\bigoplus_{d=0}^{\bar{d}}\alpha_{s,d}[t]\right),%\theta^d_{i,d+1}[t]\right),
\end{equation}
where $\bigoplus$ denotes the direct sum\footnote{The direct sum between of two matrices, i.e. $A\bigoplus B$, creates a block diagonal matrix with $A$ and $B$ on the diagonal.}. The possible values of $\alpha[t]$ are within a known set
\begin{equation}\label{eq:Psik}
{\alpha}[t]\in\Xi=\{\eta_0,\eta_1,\ldots,\eta_{q}\},
\end{equation}
where $\eta_i$ (for $i=1,\ldots,q$) denotes each possible combination, %(sampling scenario),
being $\eta_0$ the scenario without available measurements, (i.e., $\eta_0=0$). In the general case, any combination of available sensor measurement and delay is possible, leading to $q=2^{\bar n_y}-1$. $\alpha[t]$ is the result of applying a surjective function $f:\Theta\rightarrow\Xi$ on $\theta[t]$ meaning that $\alpha[t]$ is not a Markov variable neither i.i.d.

%\begin{rem}\label{rem:scenarios}
%In each real application, the function $f:\Theta\rightarrow\Xi$ gathers the possible cases, and becomes a parameter design if the sensors have some processing capabilities. If the sensor nodes collect measurements on a buffer from $t-\bar d$ to $t$ and transmit it each $\bar d+1$ periods (as in works \cite{Schenato08,Moayedi2011a,Zhu2012a}), then the set $\Xi$ will have a length of $q=2^{n_y}-1$ (with $n_y \leq \bar{n}_y$), and each $\eta_i$ will have ones in the positions related to the sensors from which the gathered delayed measurements are sent. In the multi-rate approach of \cite{Liang2010a}, each $\eta_i$ will represent the possible measurement combinations in the global instant.$\hfill\square$
%\end{rem}

\begin{exmp}\label{ej:tau}
Let us consider a system with one sensor and $\bar{d}=1$. Then $\Theta=\begin{tiny}\left\{ \begin{bmatrix}0\\0\\0\end{bmatrix}, \begin{bmatrix}1\\0\\0\end{bmatrix}, \begin{bmatrix}0\\1\\0\end{bmatrix}, \begin{bmatrix}1\\1\\0\end{bmatrix},
\begin{bmatrix}0\\0\\1\end{bmatrix}, \begin{bmatrix}1\\0\\1\end{bmatrix} \right\}\end{tiny}$. Fig.~\ref{fig_markov} illustrates the relationship between $\theta_t$ and $\theta_{t+1}$. $\theta_t=\vartheta_0$ means that $y[t]$ has not arrived at $t$ (but can still arrive, i.e. $\tau[t]>0$) and that $y[t-1]$ is lost ($\tau[t-1]>1$). $\Prb\{\theta[t+1]=\vartheta_2|\theta[t]=\vartheta_0\}=0$ because $\theta[t+1]=\vartheta_2$ would imply that $\tau[t]=0$, and $\theta[t]=\vartheta_0$ guarantees that $\tau[t]>0$. However,
\begin{scriptsize}\begin{align*}
&\Prb\{\theta[t+1]=\vartheta_1|\theta[t]=\vartheta_0\}\\
%&\Prb\{\theta[t+1]=\vartheta_1|\theta[t]=\vartheta_0\}=\\
&\qquad=\Prb\{\tau[t+1]=0,\tau[t]>1|\tau[t]>0,\tau[t-1]>1\}\\
&\qquad=\Prb\{\tau[t+1]=0,\tau[t]>1|\tau[t]>0\}\\
&\qquad=\Prb\{\tau[t+1]=0\}\Prb\{\tau[t]>1|\tau[t]>0\}\\ &\qquad=\Prb\{\alpha_{0}[t+1]=1\}\Prb\biggr\{{\|\theta_{1}[t+1]\|_1=0}\big{|}{\|\theta_{0}[t]\|_1=0}\biggl\}\\
&\qquad=\Prb\{\alpha_{0}[t+1]=1\}\Prb\{\|\theta_{1}[t+1]\|_1=0\}/\Prb\{\|\theta_{0}[t]\|_1=0\}\\
&\qquad=\beta_0\bar{\beta}_1/\bar{\beta}_0.
\end{align*}\end{scriptsize}%where it has been used Assumption~\ref{asum:iid} and Baye's theorem to retrieve~\eqref{ec:p_ij_theta}.
The full transition matrix can be obtained in the same way, leading to
\begin{tiny}
\begin{equation*}
\begin{bmatrix} \bar{\beta}_1&\beta_0\bar{\beta}_1/\bar{\beta}_0&0&0&\beta_1&\beta_0\beta_1/\bar{\beta}_0\\
0&0&\bar{\beta}_0&{\beta}_0&0&0\\
\bar{\beta}_1&\beta_0\bar{\beta}_1/\bar{\beta}_0&0&0&\beta_1&\beta_0\beta_1/\bar{\beta}_0\\
0&0&\bar{\beta}_0&{\beta}_0&0&0\\
\bar{\beta}_1&\beta_0\bar{\beta}_1/\bar{\beta}_0&0&0&\beta_1&\beta_0\beta_1/\bar{\beta}_0\\
0&0&\bar{\beta}_0&{\beta}_0&0&0
\end{bmatrix}.
\end{equation*}\end{tiny}
In this case,
\[\Xi\begin{tiny}=\left\{ \begin{bmatrix}0\\0\end{bmatrix}, \begin{bmatrix}1\\0\end{bmatrix}, \begin{bmatrix}0\\0\end{bmatrix}, \begin{bmatrix}1\\0\end{bmatrix}, \begin{bmatrix}0\\1\end{bmatrix}, \begin{bmatrix}1\\1\end{bmatrix} \right\}=\left\{ \begin{bmatrix}0\\0\end{bmatrix}, \begin{bmatrix}1\\0\end{bmatrix}, \begin{bmatrix}0\\1\end{bmatrix}, \begin{bmatrix}1\\1\end{bmatrix} \right\}\end{tiny},\]
where only the diagonal terms of $\eta_i$ have been represented.
\begin{figure}[h]
\begin{center}
  \includegraphics[width=\linewidth]{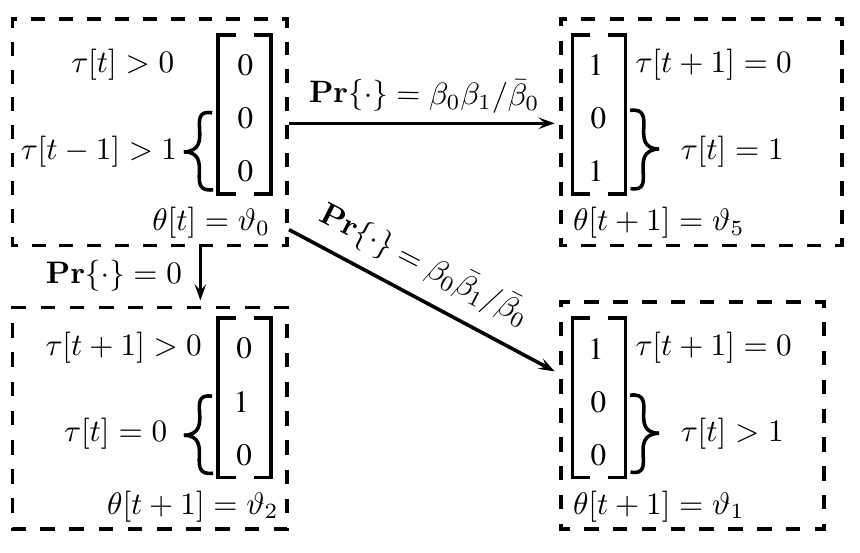}
    \caption{Markovian transitions example.}\label{fig_markov}
\end{center}
\end{figure}
$\hfill\square$
\end{exmp}

Using $\alpha[t]$, we rewrite the received measurement information at instant $t$ as
\begin{equation}\label{ec:bar_m_t}
\bar{m}[t]={\alpha}[t]\left({\bar{C}}\bar{x}[t]+\bar{v}[t]\right)
\end{equation}
with $\bar{m}[t]=[{\bar m_1[t]}^T\;\cdots\;{\bar m_{n_y}[t]}^T]^T$, $\bar{v}_t=[\bar v_{1}[t]\;\cdots\;\bar v_{\bar n_y}[t]]^T$ and $\bar v_{1}[t]=[v_{1}[t]\;\cdots\;v_{1}[t]]$. The rows of ${\bar{C}}$ are ${\bar{c}}_{s}=[{\bar{c}}_{s,0}\,\cdots\,{\bar{c}}_{s,{\bar{d}}}]^T$ with  ${\bar{c}}_{s,d}=[0_{1\times n\cdot d}\; c_{s}\;0_{1\times n\cdot(\bar d-d)}]^T$. In~\eqref{ec:bar_m_t}, $\bar{v}[t]$ is the measurement noise vector with covariance $\E\{\bar{v}[t]\bar{v}[t]^T\}=V=\bigoplus_{s=1}^{n_y}\left(\bigoplus_{{d}=0}^{\bar d}\sigma_{s}^2\right).$
%assuming a non correlated noise.

%A control instant in which all information from sensors is lost leads to $\alpha [t]=0$. If at a given control instant all the information from each sensor is available, then $\alpha[t]=I$ (as we assume delayed measurements, this means an arrival of a packet with information of each sensor from $t-\bar d$ to $t$).

\subsection{Proposed observer}
Let us represent $x[t]$ as $x_t$. We propose the following state estimation algorithm. At each instant $t$, the model is run in open loop leading to the prior estimation
\begin{equation}\label{open}
\hat{\bar{x}}_{t^-}=\bar A\,\hat{\bar{x}}_{t-1}+\bar B_u\,u_{t-1}.
\end{equation}
%where $\hat{\bar{x}}_{t^-}\triangleq\hat{\bar{x}}[t|t-1]$ and $\hat{\bar{x}}_{t-1}\triangleq\hat{\bar{x}}[t-1|t-1]$.
If no measurement is received, the best estimation of the system state is the prior estimation, i.e., $\hat{\bar{x}}_{t}=\hat{\bar{x}}_{t^-}$. Otherwise, the estimation state is updated as
\begin{equation}\label{ec:estimator_update}
\hat{\bar{x}}_{t}=\hat{\bar{x}}_{t^-}+L[t](\bar m_{t}-\alpha_t\bar C\hat{\bar{x}}_{t^-}),
\end{equation}
where $L[t]$ is the updating gain matrix.

%If some message arrives at time $t$, the state is updated as
%\begin{equation}\label{update}
%\hat{\bar{x}}_{t}=\hat{\bar{x}}_{t^-}+\sum_{\substack{\bar{s}=1}}^{\bar n_y}l_{\bar{s},t}(\bar{m}_{\bar{s},t}-\bar{C}_{\bar{s}}\hat{\bar{x}}_{t^-}),
%\end{equation}
%where $l_{\bar{s},t}$ is the updating gain that applies to the $t$-th sample of sensor $\bar{s}$ if available at the estimator node and $\bar{m}_{\bar{s},t},\,\bar{C}_{\bar{s}}$ represent the $\bar{s}$-th row of  $\bar{m}_{,t}$ and $\bar{C}$ (see~\eqref{ec:bar_m_t}). Considering a null value on the unavailable measurements, we rewrite the update equation~\eqref{update} as
%\begin{equation}
%\hat{\bar{x}}_{t}=\hat{\bar{x}}_{t^-}+L_t\alpha_t(\bar y_{t}-\bar C\hat{\bar{x}}_{t^-}),
%\end{equation}
%where $L_t$ is the updating gain matrix.

Considering~\eqref{estadosaug} and~\eqref{ec:bar_m_t}-\eqref{ec:estimator_update}, the dynamic of the estimation error, defined as
$\tilde{x}_{t}=\bar{x}_{t}-\hat{\bar{x}}_{t}$, is
\begin{equation}
\tilde{x}_t=(I-L[t]\alpha_t\bar C)\left(\bar{A}\tilde{x}_{t-1}+\bar{B}_ww_{t-1}\right)-L[t]\alpha_t \bar{v}_{t}.
\end{equation}

The aim of this work is to compute the gain matrices $L[t]$ that minimize the state estimation error while requiring low computing and storage capabilities. Thus, we propose to relate the gains with $\theta_t$ as $L[t]=L(\theta_t)$.

%Using a Kalman filter~\cite{Schenato08} leads to a time varying gain $L_t$ whose online implementation requires  $2\bar{n}^3+\bar{n}^2n_w+\bar{n}n_w^2+4\bar{n}^2\bar{n}_y+2\bar{n}\bar{n}_y^2+5\bar{n}_y^3+2\bar{n}^2+\bar{n}n_u+2\bar{n}\bar{n}_y+\bar{n}_y^2+2\bar{n}+\bar{n}_y$ (where $\bar{n}=n(\bar{d}+1)$) floating-point operations (FLOPs) per instant, leading to high computation requirements and to possible numerical problems due to the inversion of a matrix of at most $\bar{n}_y\times\bar{n}_y$.
%Using a jump linear estimator with a finite set of stored gains leads to $\bar{n}^2+\bar{n}n_u+2\bar{n}\bar{n}_y+2\bar{n}+\bar{n}_y$ FLOPs per instant, that is much lower than the previous one.

In the motivating example in~\cite{Smith2003}, the authors showed that the gains obtained with a Kalman filter depend on the history of combination of sensor availability. In the present work we extend their result to delayed measurements and multisensor transmission defining the gains as
\begin{equation}\label{def:Lt}
L(\theta_t)=\begin{cases}
   0 \: \: \text{if } f(\theta_t)=\eta_0, \text{\small{(no measurement received)}}\\
      L_i \:\:  \text{if } \theta_t=\vartheta_i,\,f(\vartheta_i)\neq\eta_0
  \end{cases}
\end{equation}
The matrices  are computed off-line leading to the finite set
\begin{equation}\label{setL}
L(\theta_t)\in\mathcal{L}=\{L_{0},\ldots,L_{r}\}.
\end{equation}
We will next show how to design such an observer when imposing constraints over $\mathcal{L}$.

\section{Observer design}\label{sec:obsvdsgn}
As the Markov chain $\{\theta_t\}$ is ergodic, it has a stationary distribution which satisfies $\pi=\pi\Lambda$. We assume the initial condition $\pi[0]=\pi$, and in consequence $\pi[t]=\pi\,\forall t$. Based on this assumption, the following theorem expresses the evolution of the state estimation error covariance matrix. %This result will be used to design the observer gains.
\begin{thm}\label{teor:Pt}
Let $P_{t-1,i}=\E\{\tilde x_{t-1}\tilde x_{t-1}^T|\theta[t-1]=\vartheta_i\}$ (with $i=1,\ldots,r$) be the covariance matrix for the state estimation error updated at the measurement instant $t-1$ with information $\theta[t-1]=\vartheta_i$ . The expected value of the covariance matrix at the measurement instant $t$, $\E\{\tilde x_{t}\tilde x_{t}^T\}$, is given by
\begin{align}\label{ec:Pt}
\sum_{j=0}^{r}\E\{\tilde x_{t}\tilde x_{t}^T|\theta[t]=\vartheta_j\}\Prb\{\theta[t]=\vartheta_j\}=\sum_{j=0}^{r}P_{t,j}\pi_j,
\end{align}
where $P_{t,j}$ is defined by
\begin{align}\label{ec:Pt_j}
\sum_{i=0}^{r}p_{i,j}\frac{\pi_i}{\pi_j}\left(F_{j}(\bar{A}P_{t-1,i}\bar{A}^T+\bar{B}_wW\bar{B}_w^T)F_{j}^T+X_{j}VX_{j}^T\right),
\end{align}
with
\begin{equation}\label{ec:F_X}
F_{j}=I-L_{j}\,f(\theta_j)\bar C,\quad X_{j}=L_{j}f(\theta_j).
\end{equation}
\end{thm}
\begin{proof} See Appendix~\ref{proof:teor_Pt}
\end{proof}
The previous theorem establishes a recursion for the covariance matrix. We thus write $\Pm_t=\mathfrak{E}\{\Pm_{t-1}\}$, where $\Pm_t\triangleq(P_{t,0},\ldots,P_{t,r})$, $\mathfrak{E}\{\cdot\}\triangleq(\mathfrak{E}_0\{\cdot\},\ldots,\mathfrak{E}_{r}\{\cdot\})$, being $\mathfrak{E}_{i}\{\cdot\}$ the linear operator that returns equation~\eqref{ec:Pt_j}. In order to compute the observer gains off-line, one must find the stable solution to the Riccati equation $\mathfrak{E}\{\Pm_{t-1}\}=\Pm_{t-1}$. In general, for cases where the observer gain depends on each state of the Markov chain, an explicit expression on the observer gain values can be found using the methods of~\cite{Smith2003} and~\cite{Han2013}. However, the methods applied in those works become untractable for the design of an observer that share the same gain for different states of the Markov chain. Hence, those methods do not directly allow to explore trade-offs between storage complexity and estimation performance. To address this issue, we adopt the following alternative optimization problem
\begin{subequations}\label{problema}
\begin{align}
&\min_{{\mathcal{L}},\Pm}\;\; \mathrm{tr}(\sum_{j=0}^{r}P_{j}\pi_j)\label{problema_1}\\
&\,\mathrm{s.t.}\;\;\; \;\mathfrak{E}\{\Pm\}-\Pm\preceq 0,\label{problema_2}
\end{align}
\end{subequations}
with $\Pm\triangleq(P_{0},\ldots,P_{r})$.

%The above formulation allows us to include constraints on the set of gains and reduce the observer complexity, just by imposing $\bar{L}_i=\bar{L}_j$ (with $i\neq j$ and $i,j=1,\ldots,\bar{r}$).
As we shall see next, the constraint in~\eqref{problema_2} is instrumental for guaranteeing boundedness of $\E\{\tilde x_{t}\tilde x_{t}^T\}$, and therefore stochastic stability. Note that the next results are independent on the constraints over ${\mathcal{L}}$.
%In this work, we intend to explore the trade-off between estimation performance versus the jump estimator complexity.

\subsection{Boundedness of the covariance}
We show in the following that if we apply the gains $\mathcal{L}$ obtained from problem~\eqref{problema}, then the sequence $\{\Pm_t\}$ (and thus $\{\E\{x_tx_t^T\}\}$) converges to the unique solution $\bar{\mathcal{P}}\triangleq(\bar{P}_1,\ldots,\bar{P}_{r})$ obtained in~\eqref{problema}.

Let us first introduce the following lemma, extended from~\cite{Sinopoli2004}, where $\bar{\mathcal{P}}\succ 0$ denotes $\bar{P}_i\succ0,\,\forall i=1\,\ldots,r$.
\begin{lem}\label{lem:OperaLin} Define the linear operator
\begin{align*}
\mathcal{T}_j(\mathcal{Y})=\sum_{i=0}^{r}p_{i,j}\frac{\pi_i}{\pi_j}F_j\bar{A}Y_i{\bar{A}}^TF_j^T
\end{align*}
where $\mathcal{T}(\cdot)\triangleq\left(\mathcal{T}_0(\cdot),\ldots,\mathcal{T}_{r}(\cdot)\right)$ and $\mathcal{Y}\triangleq\left(Y_0,\ldots,Y_{r}\right)$. Suppose that there exists $\bar{\mathcal{Y}}\triangleq\left(\bar{Y}_0,\ldots,\bar{Y}_{r}\right)\succ0$ such that $\mathcal{T}(\bar{\mathcal{Y}})\prec\bar{\mathcal{Y}}$. Then,
(a) for all $\mathcal{W}\triangleq\left(W_0,\ldots,W_{r}\right)\succeq0$, $\lim_{t\rightarrow\infty}\mathcal{T}^t(\mathcal{W})=0$\footnote{$\mathcal{T}^t\{\cdot\}$ represents the recursion of $\mathcal{T}\{\cdot\}$.};
(b) let $U\succeq0$ and consider the linear system $\mathcal{Y}_{t+1}=\mathcal{T}(\mathcal{Y}_t)+U$, initialized at $\mathcal{Y}_0$, then the sequence  $\{\mathcal{Y}_t\}$ is bounded.$\hfill\square$
\end{lem}
Using the above lemma, the following theorem proves the boundedness of $\{\Pm_t\}$.
\begin{thm}\label{teor:boundP}Under Assumption~\ref{asum:iid}, suppose that
the set $\mathcal{L}$ in~\eqref{setL} fulfills restriction~\eqref{problema_2}, i.e.,
there exists $\bar{\Pm}\succ0$ such that $\mathfrak{E}\{\bar{\Pm}\}\preceq\bar{\Pm}$. Then, for any initial condition $\Pm_0\succeq0$ the sequence $\{\Pm_t\}$ is bounded, i.e., $\{\Pm_t\}\preceq M_{\Pm}$, with $M_{\Pm}\triangleq\left(M_{P_0},\ldots,M_{P_{r}}\right)$.
\end{thm}
\begin{proof} See Appendix~\ref{proof:teor_boundP}
\end{proof}
By means of the previous theorem, the next result establishes that $\{\Pm_t\}$ converges to the solution of problem~\eqref{problema}.
\begin{thm}\label{teor:conver}Under Assumption~\ref{asum:iid}, suppose that
the set $\mathcal{L}$ in~\eqref{setL} solves problem~\eqref{problema}.
Then, for any initial condition $\Pm_0\succeq0$, the iteration $\Pm_{t+1}=\mathfrak{E}\{\Pm_t\}$ converges to the unique positive semi-definite solution $\bar{\Pm}$ obtained in problem~\eqref{problema}, i.e., $ \lim_{t\rightarrow\infty} \Pm_t=\lim_{t\rightarrow\infty}\mathfrak{E}^t\{\Pm_0\}=\bar{\Pm}\succeq0$, where $\bar{\Pm}=\mathfrak{E}\{\bar{\Pm}\}$.
\end{thm}
\begin{proof} See Appendix~\ref{proof:teor_conver}
\end{proof}
%In the next section we demonstrate how to solve problem~\eqref{problema}.
%\vspace{-0.15cm}
\subsection{Numerical issues}\label{subsec:numerical}
Problem~\eqref{problema} can be solved using the following linear matrix inequalities and bilinear equality constraints,% that makes it easy to include different constraints over observer gains.
\begin{subequations}\label{optimiza}
\begin{align}
&\min_{{\mathcal{L}},\Pm,\Rm}\;\mathrm{tr} \left(\sum_{j=0}^{r}P_{j}\pi_j\right)\\
&\begin{bmatrix} \vspace{0.1cm}P_j&\bar{\bar{M}}_j\,\bar{\bar{A}}&\bar{\bar{M}}_j\,\bar{\bar{W}}&\bar{\bar{X}}_j\,\bar{\bar{V}}\\
\vspace{0.1cm}\bar{\bar{A}}^T\,\bar{\bar{M}}_j^T&\bar{\bar{R}}&0&0\\
\vspace{0.1cm}\bar{\bar{W}}^T\,\bar{\bar{M}}_j^T&0&\bar{\bar{W}}&0\\
\bar{\bar{V}}^T\,\bar{\bar{X}}_j^T&0&0&\bar{\bar{V}}
\end{bmatrix}\succeq 0,\forall j=0,\ldots,\,r\label{LMI}\\
&\bar{\bar{P}}\,\bar{\bar{R}}=I\label{BME1}%\\
%&P=P^T\succ 0,\;Q=Q^T\succ 0,\; R=R^T,\label{PQR}
\end{align}
\end{subequations}
with
\begin{align*}
&\bar{\bar{X}}_j=\left[\sqrt{p_{0,j}\pi_{0}/\pi_j}L_{j}f(\vartheta_j)\;\cdots\;\sqrt{p_{r,j}\pi_{r}/\pi_j}L_{j}f(\vartheta_j)\right],\\
&\bar{\bar{M}}_j=\left[\sqrt{p_{0,j}\pi_{0}/\pi_j}F_{j}\;\cdots\;\sqrt{p_{r,j}\pi_{r}/\pi_j}F_{j}\right],\,\,\,\bar{\bar{A}}=\bigoplus_{i=0}^{r}\bar{A},\\
&\bar{\bar{W}}=\bigoplus_{j=0}^{r}\bar{B}_wW{\bar{B}_w}^T,\,\,\bar{\bar{V}}=\bigoplus_{j=0}^{r}V,\,\,\bar{\bar{R}}=\bigoplus_{j=0}^{r}R_j,\,\,\bar{\bar{P}}=\bigoplus_{j=0}^{r}P_j,
\end{align*}
$\Rm\triangleq({R}_1,\ldots,{R}_{r})$ and $F_j$ as defined in~\eqref{ec:F_X}. Applying extended Schur complements on~\eqref{LMI} makes problem~\eqref{problema} and~\eqref{optimiza} equivalent. %Note that it is easy to include different constraints over the observer gains, just by imposing $L_j=L_i$ (with $i\neq j$ and $i,j=1,\ldots,r$). Some possible sets $\mathcal{L}$ will be explore in Section~\ref{sec:trade}.

The optimization problem~\eqref{optimiza} is a nonconvex optimization problem because of the terms $R_j=P_j^{-1}$ in~\eqref{BME1}. We address this problem with the cone complementarity linearization algorithm (~\cite{ElGhaoui97}) over a bisection algorithm. The algorithm is omitted for brevity; an example can be found in~\cite{penarrocha2013inferential}.
%\vspace{-0.15cm}
\section{Design Trade-offs}\label{sec:trade}
In this work, we explore the trade-off between estimation performance versus jump estimator complexity.
Since the gains are related to $\theta_t$, the solution of the previous section %(without imposing any constraints on $\bar{\mathcal{L}}$)
leads to a number of non zero different gain matrices equal to\footnote{$|\mathfrak{L}|$ denotes the cardinal of the set $\mathfrak{L}$, i.e., the number of elements of $\mathfrak{L}$.} $|\mathfrak{L}|={(\bar{d}+1)!}^{n_y}((\bar{d}+2)^{n_y}-1)$ being $\mathfrak{L}$ the non zero gain matrices fulfilling
%$\mathfrak{L}\subset{\mathcal{L}}$ and
$\mathcal{L}=\mathfrak{L}\bigcup\{0\}$ (see~\eqref{setL}). We can reduce the observer complexity by imposing some equality constraints over the set $\mathcal{L}$ as ${L}_{i}={L}_{j}$ in problem~\eqref{optimiza}. % allows one to easily include these equality constraints.
Reducing the number of gains simplifies the numerical burden of~\eqref{optimiza}, as the number of decision variables are shortened. %Sharing gains has also implications on the implementation of the selection mechanism.
To implement an observer with a simple online look-up-table procedure and low storage requirements, we propose the following preconfigured sets of equalities over the possible historical measurement transmission outcomes $\Theta$ (see~\eqref{eq:Deltak}):
\begin{itemize}
\item \textbf{S1}. The observer gain is independent of the measurement scenario (cf.~\cite{Schenato08}), $|\mathfrak{L}_{S1}|=1$.% i.e., ${L}_{i}={L}_{j}$ for any pair $i\neq j$ with ${L}_{i},{L}_{j}\in\mathfrak{L}$.
\item \textbf{S2}. The observer gains depend on the number of real sensors from which measurements arrive successfully at each instant, $|\mathfrak{L}_{S2}|=n_y$.
\item \textbf{S3}. The observer gains depend on the number of real and fictitious sensors from which measurements arrive successfully at each instant, $|\mathfrak{L}_{S3}|=\bar{n}_y$.
\item \textbf{S4}. The observer gains depend on the measurement recepetion at a given instant $\alpha_t$ (see~\eqref{def:alpha}), $|\mathfrak{L}_{S4}|=2^{\bar{n}_y}-1$.
\item \textbf{S5}. The observer gains are related to the historical measurement transmission outcomes $\theta_t$, $|\mathfrak{L}_{S5}|={(\bar{d}+1)!}^{n_y}((\bar{d}+2)^{n_y}-1)$.
\end{itemize}

These gain grouping approaches, allow us to trade-off between implementation cost and estimation performance. S1 leads to the lowest cost and largest estimation error covariance, S5 gives the highest cost and best performance. The example section explores this idea.

\begin{rem} \cite{Han2013} proposed a gain that jump with the possible instant and arrival delay for each measurement in a finite set. Adapting their proposal to ours and considering Example~\ref{ej:tau}, would lead to $\mathcal{L}=\begin{tiny}\left\{ \begin{bmatrix}0&0\end{bmatrix}, \begin{bmatrix}l_1&0\end{bmatrix}, \begin{bmatrix}0&0\end{bmatrix}, \begin{bmatrix}l_1&0\end{bmatrix},
\begin{bmatrix}0&l_2\end{bmatrix}, \begin{bmatrix}l_1&l_2\end{bmatrix} \right\}\end{tiny}$, with $l_1,l_2\in\mathbb{R}^{2\times1}$ decision variables. Defining ${L}_1=\begin{bmatrix}l_1&l_2\end{bmatrix}$ and extending to the multisensor case, the method is equal to case S2.
\end{rem}

\begin{exmp}
Considering Example~\ref{ej:tau}, the proposed scenarios will impose ${\mathcal{L}}_{S1}={\mathcal{L}}_{S2}=\left\{0,L_1,0,L_1,L_1,L_1\right\}$, ${\mathcal{L}}_{S3}=\left\{0,L_1,0,L_1,L_1,L_5\right\}$, ${\mathcal{L}}_{S4}=\left\{0,L_1,0,L_1,L_4,L_5\right\}$, ${\mathcal{L}}_{S5}=\left\{0,L_1,0,L_3,L_4,L_5\right\}$.
\end{exmp}

\section{Examples}\label{sec:ej}
%In this example we analyze the different gain scheduling strategies proposed in Section~\ref{sec:trade}, as well as the relationship between number of stored gains and achieved performance.
We consider the following system (randomly chosen)
\begin{align*}
A&=\begin{bmatrix}0.73&-0.42\\0.42&0.73\end{bmatrix}+\rho,\,B_w=\begin{bmatrix}0.01&0.13\\0.01&0.08\end{bmatrix},\\
C&=\begin{bmatrix}0.53&0.39\\0.72&0.35\end{bmatrix},
\end{align*}

with $B_u=\begin{bmatrix}-0.33&0.34\end{bmatrix}^T$, and where $0\leq\rho\leq0.5$. %Note that the value of
$\rho$ makes the maximum absolute eigenvalue of $A$ (denoted by $|\lambda(A)|_{\max}$) vary between $0.8422\leq|\lambda(A)|_{\max}\leq1.5013$. The state disturbance and sensor noises covariances are %as follows
\begin{equation*}
W=\begin{bmatrix}
    0.26  & -0.003\\
   -0.003  &  0.25\\
\end{bmatrix},\;
\begin{bmatrix}
\sigma_1^2\\\sigma_2^2
\end{bmatrix}=\begin{bmatrix}
0.0086\\0.0079
\end{bmatrix}.
\end{equation*}
The measurements are independently acquired through a communication network that induces a delay that varies between 0 and 1. Thus, the amount of fictitious sensors is 4,  $|\Theta|={(1+2)!}^2=36$ (see~\eqref{eq:Deltak}), and $|\Xi|=2^4=16$ (see~\eqref{eq:Psik}).
The probabilities of delivering a measurement with a given delay are $\beta_1=\begin{bmatrix}0.32&0.22&.46\end{bmatrix}$ and $\beta_2=\begin{bmatrix}0.22&0.32&.46\end{bmatrix}$ (where $\beta_s=\begin{bmatrix}\beta_{s,0} &\cdots & \beta_{s,\bar{d}} &\bar{\beta}_{s,\bar{d}}\end{bmatrix}$, with $s=1,2$).

Let us compare the results of the implementation of the optimal Kalman filter algorithm for model~\eqref{estadosaug}-\eqref{medidasaug} (adapted from \cite{Schenato08}) and the proposed algorithm. Let us define $P=C_x\E\{\tilde{x}_t\tilde{x}_t^T\}C_x^T$, where $C_x=[I_{n}\,0_{n\times(n\cdot \bar{d})}]$ selects the covariance corresponding to $x[t]-\hat{x}[t|t]$. Then, let us introduce
\begin{equation*}
\varepsilon(\%)=\frac{\mathrm{tr}(P_{\mathrm{Kal}}-P_{\mathrm{S}})}{\mathrm{tr}(P_{\mathrm{Kal}})}\cdot100
\end{equation*}
as the factor that indicates how large the performance loss is for a given strategy S ($P_{\mathrm{S}}$) w.r.t the one obtained with the optimal Kalman filter ($P_{\mathrm{Kal}}$).

Fig.~\ref{fig1} and Table~\ref{tab1} show that performance gets worse when $|\lambda(A)|_{\max}$ increases its value. For a stable open-loop system, a good trade-off between performance and storage requirement can be to choose case S1, where a single gain leads to an estimation performance no more than 15\% worse than the optimum. However when the system is unstable, a reasonable trade off could be to choose case $S3$,  where with 4 gains the performance is at most 19\% worse than the optimum. In the present case, the Kalman filter needs at most 976 floating-point operations per instant (including matrix inversion), while the off-line methods only need 64, which implies a reduction of a 93\% in the online computing cost.

\begin{figure}[h]
\begin{center}
  \includegraphics[width=\linewidth]{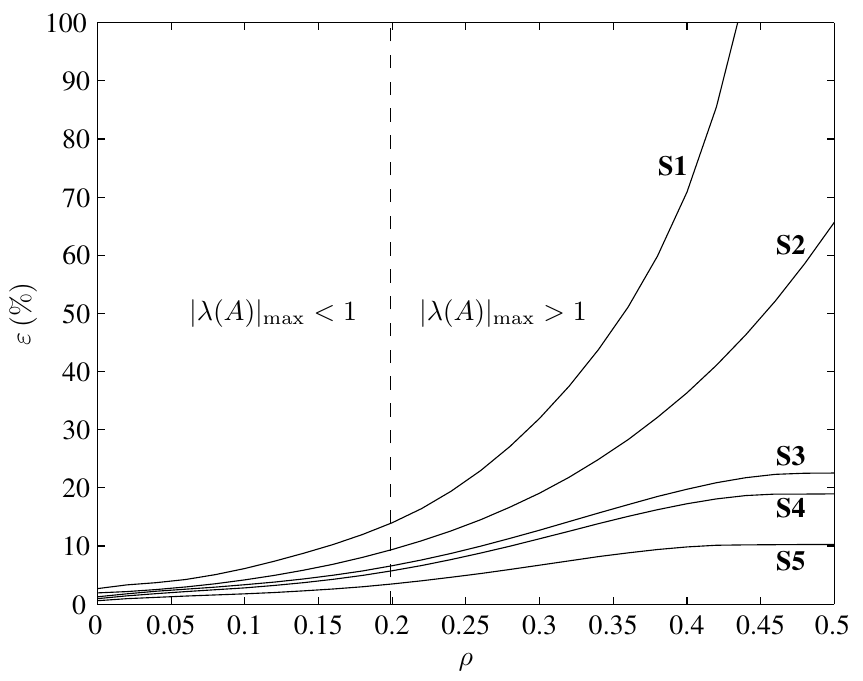}
\caption{Degradation of the estimation performance w.r.t. the optimal Kalman filter.}\label{fig1}
\end{center}
\end{figure}

\begin{table}[ht]
\begin{center}
\caption{Observers comparison.}\label{tab1}
%\begin{small}
\begin{tabular}{|l|c|c|c|c|c|}
\hline
Case & S1 & S2 & S3 & S4 & S5\\
\hline
$|\mathcal{L}|$ & 1  & 2 & 4 & 15 & 32\\
$\varepsilon$ for $\rho=0$&  2.7 &    1.9 &    1.3 & 0.9  & 0.05 \\
$\varepsilon$ for $\rho=0.5$&  231.6 &  65.7  &  22.5 & 18.9  &  10.3 \\
\hline
\end{tabular}
%\end{small}
\end{center}
\end{table}

\section{Conclusions}\label{sec:conclu}
In this work we develop a model %that accepts out-of-sequence measurements, buffered measurements on a single packet or multi-rate sensor measurements.
for multisensor networked estimation with time-varying delays and dropouts. We introduce a Markovian finite process that stores the measurement transmission outcomes on an interval, capturing the behavior of the network. Using this process, we design a jump state estimator for networked systems where its complexity can be chosen as a trade-off between estimation performance and storage requirements. The result is a finite set of gains that can be constrained to be equal for different values of the finite
measurement outcomes parameter. %states of the Markov chain.
Numerical results confirm that the computational cost of the on-line implementation can be much lower than Kalman filter approaches, while the achieved estimation performance is close to the optimum. %The performance is better than constant gain approaches at the cost of more storage requirements.

Further research may include studying Markovian delays, determining a priori the feasibility of problem~\eqref{optimiza} and analytical characterization of the performance and complexity trade-offs.
%\vspace{-0.2cm}
\section*{Acknowledgements}
This work has been funded by MICINN project number DPI2011-27845-C02-02, and grants PREDOC/2011/37 and  E-2013-02 from \textit{Universitat Jaume I}

\bibliographystyle{plain}
\bibliography{bibliografiacov}%,bibliografia}

\begin{thebibliography}{10}

\bibitem{bremau99}
P.~Br\'emaud.
\newblock {\em Markov Chains}.
\newblock Springer, New York, 1999.

\bibitem{chen2011guest}
J.~Chen, K.~H. Johansson, S.~Olariu, I.~Ch. Paschalidis, and I.~Stojmenovic.
\newblock Guest editorial special issue on wireless sensor and actuator
  networks.
\newblock {\em IEEE Trans. Autom. Control}, 56(10):2244--2246, 2011.

\bibitem{ElGhaoui97}
L.~El~Ghaoui, F.~Oustry, and M.~{AitRami}.
\newblock A cone complementarity linearization algorithm for static
  output-feedback and related problems.
\newblock {\em IEEE Trans. Autom. Control}, 42(8):1171–1176, 1997.

\bibitem{Han2013}
C.~Han, H.~Zhang, and M.~Fu.
\newblock Optimal filtering for networked systems with markovian communication
  delays.
\newblock {\em Automatica}, 49(10):3097 -- 3104, 2013.

\bibitem{hespanha2007survey}
J.~P. Hespanha, P.~Naghshtabrizi, and Y.~Xu.
\newblock A survey of recent results in networked control systems.
\newblock {\em Proceedings of the IEEE}, 95(1):138--162, 2007.

\bibitem{liu2004kalman}
X.~Liu and A.~Goldsmith.
\newblock Kalman filtering with partial observation losses.
\newblock In {\em Proc. Conf. Decision. and Control}, pages 4180--4186, 2004.

\bibitem{penarrocha2013inferential}
I.~Pe{\~n}arrocha, D.~Dolz, and R.~Sanchis.
\newblock Inferential networked control with accessibility constraints in both
  the sensor and actuator channels.
\newblock {\em International Journal of Systems Science}, 2013.

\bibitem{Penya12}
I.~Pe{\~n}arrocha, R.~Sanchis, and P.~Albertos.
\newblock Estimation in multisensor networked systems with scarce measurements
  and time varying delays.
\newblock {\em Systems \& Control Letters}, 61(4):555--562, 2012.

\bibitem{Penarrocha2012IJSS}
I.~Pe{\~n}arrocha, R.~Sanchis, and J.~A. Romero.
\newblock {State estimator for multisensor systems with irregular sampling and
  time-varying delays}.
\newblock {\em International Journal of Systems Science}, 43(8):1441--1453,
  2012.

\bibitem{Sahebsara07}
M.~Sahebsara, T.~Chen, and S.~L. Shah.
\newblock Optimal $\mathcal{H}_2$ filtering in networked control systems with
  multiple packet dropout.
\newblock {\em IEEE Trans. Autom. Control}, 52(8):1508--1513, 2007.

\bibitem{Schenato08}
L.~Schenato.
\newblock Optimal estimation in networked control systems subject to random
  delay and packet drop.
\newblock {\em IEEE Trans. Autom. Control}, 53(5):1311--1317, 2008.

\bibitem{Sinopoli2004}
B.~Sinopoli, L.~Schenato, M.~Franceschetti, K.~Poolla, M.~I. Jordan, and S.~S.
  Sastry.
\newblock {Kalman filtering with intermittent observations}.
\newblock {\em IEEE Trans. Autom. Control}, 49(9):1453--1464, 2004.

\bibitem{Smith2003}
S.~C. Smith and P.~Seiler.
\newblock {Estimation with lossy measurements : jump estimators for jump
  systems}.
\newblock {\em IEEE Trans. Autom. Control}, 48(12):2163--2171, 2003.

\end{thebibliography}
\appendix

\section{Proof of Theorem~\ref{teor:Pt}}\label{proof:teor_Pt}
Equation~\eqref{ec:Pt} is obtained using the law of total probabilities. Considering the independency between $x_{t-1}$, $\bar{v}_{t}$ and $w_{t-1}$, $P_{t,j}=\E\{\tilde x_{t}\tilde x_{t}^T|\theta_{t}=\vartheta_j\}$ can be calculated as follows.
\begin{small}\begin{align*}
&\sum_{i=0}^{r}\Prb\{\theta_{t-1}=\vartheta_i|\theta_{t}=\vartheta_j\}\E\{\tilde x_{t}\tilde x_{t}^T|\theta_{t-1}=\vartheta_i,\theta_{t}=\vartheta_j\}=\\
&=\sum_{i=0}^{r}p_{i,j}\frac{\pi_i}{\pi_j}F_j(\bar{A}\E\{\tilde{x}_{t-1}\tilde{x}_{t-1}^T|\theta_{t-1}=\vartheta_i\}{\bar{A}}^T+\E\{{w}_{t-1}{w}_{t-1}^T\})F_j^T\\
&+\sum_{i=0}^{r}p_{i,j}\frac{\pi_i}{\pi_j}X_j\E\{\bar{v}_{t}\bar{v}_{t}^T\}X_j^T
\end{align*}\end{small}which leads to~\eqref{ec:Pt_j} after using $\Prb\{\theta_{t-1}=\vartheta_i|\theta_{t}=\vartheta_j\}= \Prb\{\theta_{t}=\vartheta_j|\theta_{t-1}=\vartheta_i\}\Prb\{\theta_{t-1}=\vartheta_i\}/\Prb\{\theta_{t}=\vartheta_j\}$.

\section{Proof of Theorem~\ref{teor:boundP}}\label{proof:teor_boundP}
Considering the linear operator in Lemma~\ref{lem:OperaLin}, Theorem~\ref{teor:Pt} and constraint~\eqref{problema_2}, we have $\mathcal{T}(\bar{\Pm})\prec\mathfrak{E}\{\bar{\Pm}\}\preceq\bar{\Pm}.$ Thus, $\mathcal{T}(\cdot)$ meets the condition of Lemma~\ref{lem:OperaLin}. The evolution of $\Pm_t$ is expressed as $\Pm_{t+1}=\mathfrak{E}\{\Pm_t\}=\mathcal{T}(\Pm_t)+U.$
Since $U$ contains the disturbance and noise covariance (both positive definite and bounded), then $U\succ0$, leading that
$\{\Pm_t\}$ is bounded.

\section{Proof of Theorem~\ref{teor:conver}}\label{proof:teor_conver}
First, let us show the convergence of sequence $\{\Pm_t\}$ with initial value $\Qm_0=0$, where $\Qm_t\triangleq\left(Q_{t,0},\ldots,Q_{t,r}\right)$. Let $\Qm_t=\mathfrak{E}\{\Qm_{t-1}\}=\mathfrak{E}^t\{\Qm_0\}$, then from~\eqref{ec:Pt_j}, $\Qm_1\succeq \Qm_0=0$ and $\Qm_1=\mathfrak{E}\{\Qm_0\}\preceq\mathfrak{E}\{\Qm_1\}=\Qm_2$. By induction, $\{\Qm_t\}$ is non decreasing. Also, by Lemma~\ref{lem:OperaLin}, $\{\Qm_t\}$ is bounded and by Theorem~\ref{teor:boundP} there exists an $M_{\Qm}\triangleq\left(M_{Q_0},\ldots,M_{Q_{r}}\right)$ such that $\Qm_t\preceq M_{\Qm}$ for any $t$. Hence, the sequence converges and $\lim_{k\rightarrow\infty}\Qm_t=\bar{\Pm}\succeq0$, where $\bar{\Pm}$ is a fixed point, i.e, $\bar{\Pm}=\mathfrak{E}\{\bar{\Pm}\}$.
Second, we state  the convergence of $\Gm_t=\mathfrak{E}^k\{\Gm_0\}$,  initialized at $\Gm_0\succeq\bar{\Pm}$ where $\Gm_t\triangleq\left(G_{t,0},\ldots,G_{t,r}\right)$. Since $\Gm_1=\mathfrak{E}\{\Gm_0\}\succeq\mathfrak{E}\{\bar{\Pm}\}=\bar{\Pm}$, then $\Gm_t\succeq\bar{\Pm}$ for any $t$. Moreover
$0\preceq \Gm_{t+1}-\bar{\Pm}=\mathfrak{E}\{\Gm_t\}-\mathfrak{E}\{\bar{\Pm}\}=\mathcal{T}(\Gm_t-\bar{\Pm}). $
As $\Gm_t-\bar{\Pm}\succeq0$, following the results on Lemma~\ref{lem:OperaLin}, then $0\preceq\lim_{t\rightarrow\infty}(\Gm_t-\bar{\Pm})=0$, i.e., the sequence $\{\Gm_t\}$ converges to $\bar{\Pm}$.

We demonstrate now that for any initial condition $\Pm_0\succeq0$, the iteration $\Pm_t=\mathfrak{E}\{\Pm_{t-1}\}$ converges to $\bar{\Pm}$.
Since $0\preceq \Qm_0\preceq \Pm_0\preceq \Gm_0$, we derive by induction that $0\preceq \Qm_t\preceq \Pm_t\preceq \Gm_t$. Therefore, as $\{\Qm_t\}$ and $\{\Gm_t\}$  converge to $\bar{\Pm}$, then $\{\Pm_t\}$ also converges to $\bar{\Pm}$ and the convergence is demonstrated. Finally, we need to show that
\[
\bar{\Pm}=\arg\min_{\Pm} \mathrm{tr}\left(\sum_{j=0}^{r}P_{j}\pi_j\right) \,\, \mathrm{subject\ to}\,\, \eqref{problema_2}.%,\eqref{PQR}.
\]
Suppose this is not true, i.e. $\hat{\Pm}$ solves the optimization problem, but $\hat{\Pm}\neq\mathfrak{E}\{\hat{\Pm}\}$. Since $\hat{\Pm}$  is a feasible solution, then $\hat{\Pm}\succ\mathfrak{E}\{\hat{\Pm}\}=\hat{\hat{\Pm}}$. However, this implies $\mathrm{tr}\left(\sum_{j=0}^{r}\hat{P}_{j}\pi_j\right)>\mathrm{tr}\left(\sum_{j=0}^{r}\hat{\hat{P}}_{j}\pi_j\right)$, which contradicts the hypothesis of optimality of matrix $\hat{\Pm}$. Therefore $\hat{\Pm}=\mathfrak{E}\{\hat{\Pm}\}$. Furthermore $\bar{\Pm}$ is unique since for a set of observer gains such that
\[
[\bar{\Pm},\,\,{\mathcal{L}}]=\arg\min_{\Pm,\mathcal{L}} \mathrm{tr}\left(\sum_{j=0}^{r}{P}_{j}\pi_j\right) \,\, \mathrm{subject\ to}\,\, \eqref{problema_2},
\]
we have shown that the sequence converges to $\bar{\Pm}$, and this concludes the theorem.
\end{document}